\documentclass[12pt]{journal}
\usepackage{makeidx}
\usepackage{multirow}
\usepackage{multicol}
\usepackage[dvipsnames,svgnames,table]{xcolor}
\usepackage{graphicx}
\usepackage{epstopdf}
\usepackage{ulem}
\usepackage{fontenc}
\usepackage{mathrsfs}
\usepackage{float}
\usepackage{curves}
\usepackage{caption}
\usepackage{hyperref}
\usepackage{amsmath,amsfonts,amssymb,amscd,amsmath,enumerate,verbatim}
\usepackage{amssymb}
\usepackage{algorithm}
\usepackage{tabularx}
\usepackage{multirow}
\usepackage[capitalize]{cleveref} 
\usepackage{algpseudocode}

\usepackage{enumerate,rotating,lineno,setspace,float}
\usepackage[small,compact]{titlesec}
\usepackage{amsmath,amsfonts,amssymb,amscd,amsthm,xspace}

\usepackage[paperwidth=560pt,paperheight=700pt,top=72pt,right=67pt,bottom=72pt,left=72pt]{geometry}
\newtheorem{theorem}{Theorem}[section]

\newtheorem{observation}[theorem]{Observation}

\newtheorem{Example}[theorem]{Example}
\newtheorem{definition}[theorem]{Definition}

\newtheorem{question}[theorem]{Question}

\def\({\left(}
\def\){\right)}
\def\[({\left[}
\def\]{\right]}

\def\diam{{\rm diam}}

\def\rn{{\rm rn}}
\def\span{{\rm span}}

\def\dis{\displaystyle}
\def\ve{\varepsilon}

\newcommand{\be}{\begin{equation}}
\newcommand{\ee}{\end{equation}}
\newcommand{\ben}{\begin{equation*}}
\newcommand{\een}{\end{equation*}}
\newcommand{\bea}{\begin{eqnarray}}
\newcommand{\eea}{\end{eqnarray}}
\newcommand{\bean}{\begin{eqnarray*}}
\newcommand{\eean}{\end{eqnarray*}}

\begin{document}
\thispagestyle{empty}
\setcounter{page}{01}
\begin{center}
\textbf{Some techniques to find large lower bound trees for the radio number}
\end{center}
\markboth{\it South East Asian J. of Mathematics and Mathematical Sciences}{\it Some techniques to find large lower bound trees for the radio number}
\begin{center}
{\bf Devsi Bantva}\footnote{Corresponding author.}\vspace{.2cm}\\
Lukhdhirji Engineering College - Morvi - 363 642, Gujarat (India)\vspace{.2cm}\\
E-mail : devsi.bantva@gmail.com \vspace{0.6cm}\\
{\bf P. L. Vihol}\vspace{.2cm}\\
Government Engineering College - Gandhinagar - 382 028, Gujarat (India)\vspace{.2cm}\\
E-mail : viholprakash@yahoo.com \vspace{0.6cm}\\
\end{center}\vspace{.1cm}

{\bf Abstract:} For a simple finite connected graph $G$, let $\diam(G)$ and $d_{G}(u,v)$ denote the diameter of $G$ and distance between $u$ and $v$ in $G$, respectively. A radio labeling of a graph $G$ is a mapping $f$ : $V(G)$ $\rightarrow$ \{0, 1, 2,...\} such that $|f(u)-f(v)|\geq \diam(G) + 1 - d_{G}(u,v)$ holds for every pair of distinct vertices $u,v$ of $G$. The radio number $\rn(G)$ of $G$ is the smallest number $k$ such that $G$ has radio labeling $f$ with max$\{f(v):v \in V(G)\}$ = $k$. Bantva {\it et al.} gave a lower bound for the radio number of trees in \cite[Lemma 3.1]{Bantva2} and, a necessary and sufficient condition to achieve this lower bound in \cite[Theorem 3.2]{Bantva2}. Denote the lower bound for the radio number of trees given in \cite[Lemma 3.1]{Bantva2} by $lb(T)$. A tree $T$ is called a lower bound tree for the radio number if $\rn(T)$ = $lb(T)$. In this paper, we construct some large lower bound trees for the radio number using known lower bound trees. \vspace{0.2cm}\\
{\bf Keywords and Phrases:} Interference, channel assignment, radio labeling, radio number, tree.  \vspace{0.2cm}\\
{\bf 2020 Mathematics Subject Classification:} 05C78, 05C15, 05C12. \vspace{0.2cm}\\

\section{Introduction}
The channel assignment problem is the task of allocating channels (non-negative integers) to each TV or radio transmitter such that the interference constraints are satisfied and channels use is kept to a minimum. In 1980, Hale looked for a mathematical solution of this channel assignment problem using a graph model in \cite{Hale}. In a graph model, a set of vertices represents transmitters; two vertices are adjacent if the corresponding transmitters are very close and distance two apart if corresponding transmitters are close. In 1988, Roberts in a private communication with Griggs proposed that close transmitters must receive different channels and very close transmitters must receive channels that are at least two apart. Motivated by this, Griggs and Yeh \cite{Griggs} introduced the following $L(2,1)$-labeling problem: An {\it $L(2,1)$-labeling} of a graph $G=(V(G),E(G))$ is a function $f$ from the vertex set $V(G)$ to the set of non-negative integers such that $|f(u)-f(v)|\geq2$ if $d(u,v)=1$ and $|f(u)-f(v)|\geq1$ if $d(u,v)=2$. The {\it span} of $f$ is defined as $\max\{|f(u)-f(v)| : u, v \in V(G)\}$, and the minimum span taken over all $L(2,1)$-labelings of $G$ is called the {\it $\lambda$-number} of $G$, denoted by $\lambda(G)$. This graph labeling is also known as distance two labeling as the conditions imposed on vertices are within two distance. The readers are requested to refer \cite{Calamoneri} and \cite{Yeh1} for different directions of work and recent results on $L(2,1)$-labeling.

The distance conditions on vertices were later extended to include pairs of vertices at distances larger than two to the maximum possible distance in graph - the diameter of a graph. Denote by $\diam(G)$ the {\it diameter} of $G$, that is, the maximum distance among all pairs of vertices in $G$. Chartrand {\it et al.} \cite{Chartrand1} introduced the concept of radio labeling as follows.

\begin{definition}
A radio labeling of a graph $G$ is a mapping $f: V(G) \rightarrow \{0,1,2,\ldots\}$ such that for every pair of distinct vertices $u, v$ of $G$,
\begin{equation}\label{def:rn}
d_{G}(u,v) + |f(u)-f(v)| \geq \diam(G) + 1.
\end{equation}
The span of $f$ is defined as $\span(f) = \max \{f(v) : v \in V(G)\}$. The radio number of $G$ is defined as
$$
\rn(G) := \min\{\span(f) : f \mbox{ is a radio labeling of }G\}
$$
with minimum taken over all radio labelings $f$ of $G$. A radio labeling $f$ of $G$ is optimal if $\span(f) = \rn(G)$.
\end{definition}

A radio labeling problem is a min-max type optimization problem. Note that any optimal radio labeling must assign $0$ to some vertex and also in the case $\diam(G) = 2$, we have $\rn(G) = \lambda(G)$. Observe that any radio labeling should assign different labels to distinct vertices. In fact, a radio labeling induces a linear order $u_{0},u_{1},\ldots,u_{n-1}$ ($n$ = $|V(G)|$) of vertices of $G$ such that $0 = f(u_{0}) < f(u_{1}) < ... < f(u_{n-1}) = \span(f)$.

A radio labeling is considered in one of the most tough graph labeling problems. The radio number of graphs is known for very few graph families. The summary of results on the radio number of graphs can be found in a survey article \cite{Chartrand}. Recently, the radio number of trees remain the focus of many researchers. The upper bound for the radio number of paths was given by Chartrand {\it et al.} in \cite{Chartrand2} while the exact radio number for paths is determined by Liu and Zhu in \cite{Liu}. In \cite{Daphne1}, Liu gave a lower bound for the radio number of trees and presented a class of trees namely spiders achieving this lower bound. In \cite{Li}, Li {\it et al.} determined the radio number of complete $m$-ary trees. In \cite{Tuza}, Hal\'asz and Tuza determined the radio number of level-wise regular trees. In \cite{Bantva2}, Bantva {\it et al.} gave a lower bound for the radio number of trees which is same as the one given by Liu in \cite{Daphne1} but using slightly different notations. They gave a necessary and sufficient condition to achieve this lower bound and presented three classes of trees namely banana trees, firecrackers trees and a special class of trees achieving this lower bound. Recently, Chavez {\it et al.} also discussed the radio number of trees and gave some methods to find large lower bound trees in \cite{Chavez}.

Denote the lower bound for the radio number of trees given in \cite[Lemma 3.1]{Bantva2} by $lb(T)$. A tree $T$ is called lower bound tree for the radio number if $\rn(T) = lb(T)$. In this paper, our purpose is to give some large lower bound trees for the radio number using known lower bound trees. We  construct three families of trees, namely $T_{w_{k}}$, $T_{S_{k}}$ and $T_{D_{k}}$(see Section \ref{Main:sec} for definition and detail) obtained by taking graph operations on a given tree $T$ or a family of trees $T_{i}, 1 \leq i \leq k$ whose radio number is equal to the lower bound given in \cite[Lemma 3.1]{Bantva2}. The radio numbers are related with following relations:
\begin{enumerate}[\rm (1)]
  \item $\rn(T_{w_{k}}) = \displaystyle\sum_{i=1}^{k} (\rn(T_{i})+(n_{i}-1)(d-d_{i}))-k+1$.
  \item $\rn(T_{S_{k}}) = k(\rn(T)+n_{0}(d-d_{0}-2)+d_{0})+1$.
  \item $\rn(T_{D_{k}}) = 2k(\rn(T)+n_{0}(d-d_{0}-3)+d_{0})+d$.
\end{enumerate}
where $\diam(T_{x})$ = $d$,  $|T|$ = $n_{0}$, $\diam(T)$ = $d_{0}$, $|T_{i}|$ = $n_{i}$ and $\diam(T_{i})$ = $d_{i}$ for $1 \leq i \leq k$ and $x \in \{w_{k},S_{k},D_{k}\}$.

\section{Preliminaries}\label{Preli:sec}

In this section, we define necessary terms and also present some known results which will be used in the present work. We follow \cite{West} for standard graph theoretic terms and notations. The {\it distance} $d_{G}(u,v)$ between two vertices $u$ and $v$ is the length of a shortest path joining $u$ and $v$ in $G$. The {\it diameter} of a graph $G$ is max\{$d_{G}(u,v)$ : $u, v \in V(G)$\}. A {\it tree} $T$ is a connected graph that contains no cycle. For a tree $T$, denote {\it vertex set} and {\it edge set} by $V(T)$ and $E(T)$. A $k$-star $S_k$ is a tree consisting of $k$-leaves and another vertex joined to all leaves by edges. A path $P_m$ on $m$ vertices is a tree in which every vertex has degree at most two. The complete $m$-ary tree of height $h$, denoted by $T_{h,m}$, is a rooted tree such that each vertex other than leaves (degree-one vertices) has $m$ children and all leaves are distance $h$ apart from the root. The level-wise regular tree is a tree rooted at one vertex $w$ or two adjacent vertices $w$ and $w'$, in which all vertices with the minimum distance $i$ from $w$ or $w'$ have the same degree $m_i$ for $0 \leq i \leq h$, where $h$ is the height of $T$. Denote these trees by $T^1 = T^1_{m_0,m_1,\ldots,m_{h-1}}$ with one root and $T^2 = T^2_{m_0,m_1,\ldots,m_{h-1}}$ with two roots, respectively. The $(m,k)$-banana tree, denoted by $B(m,k)$, is a tree obtained by joining one leaf of each of $m$ copies of a $(k-1)$-star to a single root (which is distinct from all vertices in the $k$-stars). The $(m,k)$-firecrackers tree, denoted by $F(m,k)$, is the tree obtained by taking $m$ copies of a $(k-1)$-star and identifying a leaf of each of them to a vertex of $P_m$. A tree is called a caterpillar if the removal of all its degree-one vertices results in a path, called the spine. Denote by $C(m,k)$ the caterpillar in which the spine has length $m-3$ and all vertices on the spine have degree $k$.

In \cite{Daphne1}, the weight of $T$ from $v \in V(T)$ is defined as $w_{T}(v)$ = $\sum_{u \in V(T)}d_{T}(u,v)$ and the weight of $T$ as $w(T)$ = $\min\{w_{T}(v) : v \in V(T)\}$. A vertex $v \in V(T)$ is a weight center of $T$ if $w_{T}(v)$ = $w(T)$. Denote the set of weight center(s) by $W(T)$. In \cite{Daphne1}, author viewed a tree $T$ rooted at a weight center $w$ and defined the {\it level function} on $V(T)$ from fix root $w$ by $L_{w}(u)$ = $d_{T}(w,u)$ for any $u \in V(T)$. For any two vertices $u$ and $v$, if $u$ is on the $(w,v)$-path ($w$ is a weight center), then $u$ is an {\it ancestor} of $v$, and $v$ is a {\it descendent} of $u$. If $u$ is a neighbour of a weight center $w$ then the subtree induced by $u$ together with all its descendants is called a {\it branch} at $u$. Two branches are called {\it different} if they are induced by two different vertices adjacent to the same weight center $w$. Using these terms and notations, Liu presented the following result in \cite{Daphne1}.
\begin{theorem}\cite{Daphne1} Let $T$ be an $n$-vertex tree with diameter $d$. Then
\begin{equation}\label{eq:lb0}
\rn(T) \geq (n-1)(d+1)+1-2w(T).
\end{equation}
Moreover, the equality holds if and only if for every weight center $w^{*}$, there exists a radio labeling $f$ with $0 = f(u_{0}) < f(u_{1}) < \ldots < f(u_{n-1})$, where all the following hold (for all $0 \leq i \leq n-2$);
\begin{enumerate}[\rm (1)]
\item $u_{i}$ and $u_{i+1}$ are in different branches (unless one of them is $w^{*}$);
\item $\{u_{0},u_{n-1}\} = \{w^{*},v\}$, where $v$ is some vertex with $L_{w^{*}}(v) = 1$;
\item $f(u_{i+1}) = f(u_{i})+d+1-L_{w^{*}}(u_{i})-L_{w^{*}}(u_{i+1})$.
\end{enumerate}
\end{theorem}

In \cite{Bantva2}, Bantva {\it et al.} gave a lower for the radio number of trees which is same as one given in \cite{Daphne1} but using different notations and also gave necessary and sufficient condition to achieve the lower bound. They viewed a tree $T$ rooted at $W(T)$: if $W(T) = \{w\}$, then $T$ is rooted at $w$; if $W(T) = \{w, w'\}$ (where $w$ and $w'$ are adjacent), then $T$ is rooted at $w$ and $w'$ in the sense that both $w$ and $w'$ are at level 0. They called two branches are {\it different} if they are at two vertices adjacent to the same weight center (which is same as in \cite{Daphne1}), and {\it opposite} if they are at two vertices adjacent to different weight centers. The later case occurs only when $T$ has two weight centers. They defined the {\it level of $u$} in $T$ as
\begin{equation}
L_{T}(u) := \mbox{min}\{d_{T}(u,x) : x \in W(T)\}, u \in V(T)
\end{equation}
and the {\it total level of $T$} as
\begin{equation}
L(T) := \displaystyle\sum_{u \in V(T)} L_{T}(u).
\end{equation}
Define
\begin{eqnarray*}
\ve(T) & = &\left\{
\begin{array}{l}
\begin{tabular}{lll}
1, & \mbox{ if $T$ has only one weight center}, \\
0, & \mbox{ if $T$ has two (adjacent) weight centers}.
\end{tabular}
\end{array}
\right.
\end{eqnarray*}

Using these terms and notation, Bantva {\it et al.} gave a lower bound for the radio number of trees in \cite{Bantva2} as follows.
\begin{theorem}\cite{Bantva2}\label{thm:lb}
Let $T$ be a tree with order $n$ and diameter $d \ge 2$. Denote $\ve = \ve(T)$. Then
\begin{equation}\label{eq:lb}
\rn(T) \ge (n-1)(d+\ve) - 2 L(T) + \ve.
\end{equation}
\end{theorem}

The readers should note that both lower bounds in \eqref{eq:lb0} and \eqref{eq:lb} are identical because when $W(T)$ = \{$w$\} then $L(T)$ = $w(T)$ and when $W(T)$ = \{$w,w'$\} then $L(T)$ = $w(T)+n/2$. Hence in our further discussion, we denote the identical right-hand side of \eqref{eq:lb0} and \eqref{eq:lb} by $lb(T)$ for a given tree $T$. The next result is a necessary and sufficient condition for $\rn(T) = lb(T)$ given in \cite{Bantva2}.

\begin{theorem}\cite{Bantva2}\label{thm:ub}
Let $T$ be a tree with order $n$ and diameter $d \ge 2$. Denote $\ve = \ve(T)$. Then
\begin{equation}\label{eq:ub}
\rn(T) = (n-1)(d+\ve) - 2 L(T) + \ve
\end{equation}
holds if and only if there exists a linear order $u_{0},u_{1},\ldots,u_{n-1}$ of the vertices of $T$ such that
\begin{enumerate}[(a)]
\item $u_{0} = w$ and $u_{n-1} \in N(w)$ when $W(T) = \{w\}$, and $\{u_{0}, u_{n-1}\} = \{w,w'\}$ when $W(T) = \{w,w'\}$;
\item the distance $d_{T}(u_{i}, u_{j})$ between $u_{i}$ and $u_{j}$ in $T$ satisfies $(0 \leq i < j \leq n-1)$
\end{enumerate}
\begin{equation}\label{eq:dij}
d_{T}(u_{i},u_{j}) \geq \displaystyle\sum_{t = i}^{j-1} (L_{T}(u_{t})+L_{T}(u_{t+1})) - (j - i)(d+\ve) + (d+1).
\end{equation}
Moreover, under this condition the mapping $f$ defined by
\begin{equation}\label{eq:f0}
f(u_{0}) = 0
\end{equation}
\begin{equation}\label{eq:f}
f(u_{i+1}) = f(u_{i}) - L_{T}(u_{i+1}) - L_{T}(u_{i}) + (d + \ve),\;\, 0 \leq i \leq n-2
\end{equation}
is an optimal radio labeling of $T$.
\end{theorem}

A tree $T$ for which $\rn(T)$ is given by \eqref{eq:ub} is called a {\it lower bound tree}. We denote the set of all lower bound trees by $\mathcal{T}_{lb}$. Hence, $\mathcal{T}_{lb}$ = \{$T$ : $\rn(T)$ = $lb(T)$\}. Our aim is to add more and more trees $T$ in the set $\mathcal{T}_{lb}$. Some known trees which are members as well as non-members of this set are as follows (see Table \ref{Table1} also). In \cite{Liu}, Liu and Zhu determined the radio number of paths. It is easy to check that $P_{2k}$ are lower bound paths while $P_{2k+1}$ are not lower bound paths. In \cite{Li}, Li {\it et al.} gave the radio number of complete $m$-ary trees ($m \geq 3$) which are lower bound trees. However, the complete binary trees whose radio number is also determined in \cite{Li} are not lower bound trees. In \cite{Tuza}, Hal\'asz and Tuza determined the exact radio number of complete level-wise regular trees with  all non-leaf vertices of degree more than two which are lower bound trees. In \cite{Bantva2}, Bantva {\it et al.} determined the radio number of banana trees and firecrackers trees which both are lower bound trees. The authors also determined the radio number of $C(m,k)$ in \cite{Bantva2} and it is easy to show that $C(m,k)$ are lower bound trees when $m$ is even and non-lower bound trees when $m$ is odd. It is interesting and challenging task to find lower bound trees of more complex structure. Note that even an addition or a deletion of a vertex or an edge make  lower bound tree to non-lower bound tree and vice-versa. For example, it is known that a path $P_{2k+1}\;(k \geq 1)$ is not a lower bound tree but deletion of one leaf vertex from a path $P_{2k+1}$ makes a path $P_{2k}$ which is a lower bound tree. Similarly, the converse procedure of above for a path $P_{2k}$ makes a lower bound tree to non lower bound tree.

Further recall that a tree represents a network of transmitters and often such a network of transmitters is expanded then it is expected that the resultant graph is to be a tree due to simple structure properties of trees. Moreover, the network operator additionally wants that the large tree obtained by expansion of network is to be a lower bound tree as the lower bound tree minimize the spectrum of channels. This motivated us to find more lower bound trees of complex structure. Usually trees of more complex structure are constructed using some graph operation on given tree $T$ or a family of trees $T_{i}, i=1,2,...,k$.  Hence, it raises the following question.

\begin{question} How to find large lower bound tree using a known lower bound tree $T$ or a family of lower bound trees $T_i, i=1,2,\ldots,k$ ?
\end{question}

Since the question is unlikely to answer completely, the task remains to present examples and constructions to find lower bound trees to answer the above question. In the next section, we give three constructions to find large lower bound trees from given lower bound tree $T$ or a family of lower bound trees $T_i,\;i=1,2,\ldots,k$.

\section{Main results}\label{Main:sec}
We consider a tree $T$ of order $n_{0}$ and diameter $d_{0}$ with weight center $w_{0}$. In case of a family of trees, we consider trees $T_{i}$ ($1 \leq i \leq k$) of order $n_{i}$ and diameter $d_{i}$ with single weight center $w_{i}$. If $T_{x}$ is any tree obtained by taking graph operation on $T$ or a family of trees $T_{i}$ then we take $|T_{x}|$ = $n$ and diam($T_{x}$) = $d$. Let $k \geq 2$ be an integer. A $k$-star $S_{k}$ is a tree consisting of $k$ leaves and another vertex joined to all leaves by edges. A $k$-double star $D_{k}$ is a tree which is formed by joining $k$ edges to each of the two vertices of $K_{2}$. It is known that both $k$-star $S_k$ and $k$-double star $D_k$ are lower bound trees. Let $*_{1}$ be the graph operation which identifies weight centers $w_{i}$ of trees $T_{i}$, $1 \leq i \leq k$ with a single vertex $w$. Denote the tree obtained by taking graph operation $*_{1}$ on trees $T_{i}$, $1 \leq i \leq k$ by $T_{w_{k}}$. Note that the weight center of $T_{w_{k}}$ is $w$ and $|T_{w_{k}}|$ = $\sum_{i=1}^{k}n_{i}-k+1$. Let $T$ and $T'$ be two trees such that $W(T)$ = \{$w_0$\} and $T'$ has $m$ leaves. Let $*_{2}$ be the graph operation which identifies a weight center $w_0$ of a copy of tree $T$ at each leaf of $T'$ denoted by $T_{T'}$. Let $T_{S_k}$ and $T_{D_k}$ denote the tree obtained by taking graph operation $*_{2}$ of a tree $T$ with $k$-star $S_k$ and $k$-double star $D_k$, respectively. Note that $|W(T_{S_k})|=1$ and $|W(T_{D_k})|=2$. Moreover, if $w$ is the vertex adjacent to all leaves in $S_k$ and $w_1,w_2$ are two adjacent vertices adjacent to all leaves in $D_k$ then $W(T_{S_k}) = \{w\}$ and $W(T_{D_k}) = \{w_1,w_2\}$. It is clear that  $|T_{S_{k}}|$ = $kn_{0}+1$ and $|T_{D_{k}}|$ = $2(kn_{0}+1)$.

\begin{observation} Let $T_{w_{k}}$, $T_{S_{k}}$ and $T_{D_{k}}$ be defined as above then the following hold.
\begin{enumerate}[\rm (a)]
  \item $|T_{w_{k}}|$ = $\sum_{i=1}^{k}n_{i}-k+1$, $|T_{S_{k}}|$ = $kn_{0}+1$ and $|T_{D_{k}}|$ = $2(kn_{0}+1)$.
  \item $\diam(T_{w_{k}}) \geq \diam(T_{i})$ $(i=1,\ldots,k)$, $\diam(T_{S_{k}}) \geq \diam(T)+2$ and $\diam(T_{D_{k}}) \geq \diam(T)+3$.
  \item For any $u \in V(T_{w_{k}})$, $L_{T_{w_{k}}}(u)$ = $L_{T_{i}}(u)$.
  \item For any $u \in V(T_{S_{k}})$, $L_{T_{S_{k}}}(u)$ = $L_{T}(u)+1$.
  \item For any $u \in V(T_{D_{k}})$, $L_{T_{D_{k}}}(u)$ = $L_{T}(u)+1$.
\end{enumerate}
\end{observation}

\begin{theorem}\label{thm:Twk}\footnote{We come to know that the similar type of result is also appeared in \cite{Chavez} without use of Theorem \ref{thm:ub}, however this result was published earlier than \cite{Chavez} in \cite{Bantva3} without proof (as an extended abstract form) and thus we included it here.} If $T_{i} \in \mathcal{T}_{lb}$, $1 \leq i \leq k$ then $T_{w_{k}} \in \mathcal{T}_{lb}$ and \begin{equation}\label{rn:Twk}
\rn(T_{w_{k}}) = \displaystyle\sum_{i=1}^{k} (\rn(T_{i})+(n_{i}-1)(d-d_{i}))-k+1.
\end{equation}
\end{theorem}
\begin{proof} We prove that $\rn(T_{w_{k}})$ = $lb(T_{w_{k}})$ and for this purpose it is enough to give a linear order $u_{0},u_{1},...,u_{n-1}$ of vertices of $T_{w_{k}}$ which satisfies the conditions of Theorem \ref{thm:ub}.

Since each $T_{i} \in \mathcal{T}_{lb}$ ($1 \leq i \leq k$), the radio number of each $T_{i}, i=1,2,\ldots,k$ is given by
\begin{equation}\label{rn1} \rn(T_{i}) = (n_{i}-1)(d_{i}+1)-2L(T_{i})+1.\end{equation}
Moreover, by Theorem \ref{thm:ub}, there exists a linear order $u^{i}_{0},u^{i}_{1},...,u^{i}_{n_{i}-1}$ of $V(T_{i})$, $1 \leq i \leq k$ which satisfies the conditions (a) $u^{i}_{0}$ = $w_{i}$ and $u^{i}_{n_{i}-1} \in N(w_{i})$ and (b) $d_{T_{i}}(u^{i}_{l},u^{i}_{m}) \geq \sum_{t=l}^{m-1}(L_{T_{i}}(u_{t})+L_{T_{i}}(u_{t+1}))-(l-m)(d_{i}+1)+(d_{i}+1), 0 \leq l < m \leq n_{i}-1$. The radio labeling defined by \eqref{eq:f0}-\eqref{eq:f} is an optimal radio labeling whose span is the right-hand side of \eqref{rn1}.

Now we define a linear order $u_{0},u_{1},...,u_{n-1}$ of vertices of $T_{w_{k}}$ as follows: Let $u_{0}$ = $w$ and for all other $u_{t}, 1 \leq t \leq n-1$ we use the following Algorithm 1.
\begin{algorithm}
\caption{A linear order $\vec{u} := \{u_0,u_1,\ldots,u_{n-1}\}$ of $V(T_{w_{k}})$.}
\hspace*{\algorithmicindent} \textbf{Input:} A list of linear orders $u^{i}_{0},u^{i}_{1},\ldots,u^{i}_{n_{i}-1}$ of $V(T_{i})$, $i=1,2,\ldots,k$ and $w$.
\begin{algorithmic}[1]
\State $u_{0} \leftarrow w$
\State $n_{0} \leftarrow 0$
\For{\texttt{$1 \leq i \leq k$}}
\For{\texttt{$1 \leq j \leq n_i-1$}}
\State $t \leftarrow j+\displaystyle\sum_{z=1}^{i}n_{z-1}$
\State $u_{t} \leftarrow u^{i}_{j}$
\EndFor
\EndFor
\State \textbf{return} $\vec{u}:=\{u_0,u_1,\ldots,u_{n-1}\}$
\end{algorithmic}
\hspace*{\algorithmicindent} \textbf{Output:} A linear order $\vec{u} = \{u_0,u_1,\ldots,u_{n-1}\}$.
\end{algorithm}

Then $u_{n-1} \in N(w)$ and for each $0 \leq i \leq n-1$, $u_{i}$ and $u_{i+1}$ are in different branches.

\textsf{Claim:} The linear order \{$u_{0},u_{1},...,u_{n-1}$\} satisfies \eqref{eq:dij}.

Let $u_{l}, u_{m}$ be two arbitrary vertices. Without loss of generality, we assume $l-m \geq 2$. We denote the right-hand side of \eqref{eq:dij} by $S_{i,j}$ for simplicity. If $u_{l},u_{m} \in V(T_{i})$ then $u_{l}$ = $u_{x}^{i}$ and $u_{m}$ = $u_{y}^{i}$ for some $i$. Note that $d_{T_{w_{k}}}(u_{l},u_{m})$ = $d_{T}(u_{x}^{i},u_{y}^{i})$ and $u_{t} \in V(T_{i})$ for $l \leq t \leq m$. Hence, we have $S_{i,j}$ = $\sum_{t=l}^{m-1}(L_{T_{w_{k}}}(u_{t})+L_{T_{w_{k}}}(u_{t+1}))-(l-m)(d+1)+(d+1) \leq \sum_{t=l}^{m-1}(L_{T_{w_{k}}}(u_{t})+L_{T_{w_{k}}}(u_{t+1}))-(l-m)(d_{i}+1)+(d_{i}+1) = \sum_{t=l}^{m-1}(L_{T_{i}}(u_{t}^{i})+L_{T_{i}}(u_{t+1}^{i}))-(l-m-1)(d_{i}+1) \leq d_{T_{i}}(u_{x}^{i},u_{y}^{i}) = d_{T_{w_{k}}}(u_{l},u_{m})$. If $u_{l} \in V(T_{i})$ and $u_{m} \in V(T_{j})$ then $u_{l}$ = $u_{x}^{i}$ and $u_{m}$ = $u_{y}^{j}$. Note that $d_{T_{w_{k}}}(u_{l},u_{m})$ = $L_{T_{w_{k}}}(u_{l})+L_{T_{w_{k}}}(u_{m})$ as $u_{l}$ and $u_{m}$ are in different branches. Let $\alpha$ = max\{$L_{T_{w_{k}}}(u_{a}) : l < a < m$\} then $2\alpha \leq d$ and hence $S_{i,j}$ = $\sum_{t=l}^{m-1}(L_{T_{w_{k}}}(u_{t})+L_{T_{w_{k}}}(u_{t+1}))-(l-m)(d+1)+(d+1) = L_{T_{w_{k}}}(u_{l})+L_{T_{w_{k}}}(u_{m})+2\sum_{t=l+1}^{m-1}L_{T_{w_{k}}}(u_{t})-(l-m-1)(d+1) \leq L_{T_{w_{k}}}(u_{l})+L_{T_{w_{k}}}(u_{m})-(l-m-1)(d+1-2\alpha) \leq L_{T_{w_{k}}}(u_{l})+L_{T_{w_{k}}}(u_{m}) = d_{T_{w_{k}}}(u_{l},u_{m})$.

Thus a linear order $u_{0},u_{1},...,u_{n-1}$ of vertices of $T_{w_{k}}$ satisfies the conditions of Theorem \ref{thm:ub}. The radio number for $T_{w_{k}}$ is given by the right-hand side of \eqref{eq:ub} for which $n$ = $n_{1}+n_{2}+...+n_{k}-k+1$, $L(T_{w_{k}})$ = $L(T_{1})+L(T_{2})+...+L(T_{k})$ and using \eqref{rn1} we have

\bean
\rn(T_{w_{k}}) & = & (n-1)(d+1)-2L(T_{w_{k}})+1 \\
& = & \Bigg(\displaystyle\sum_{t=1}^{k}n_{i}-k\Bigg)(d+1)-2\Bigg(\displaystyle\sum_{t=1}^{k}L(T_{i})\Bigg)+1 \\
& = & \displaystyle\sum_{t=1}^{k}(n_{i}-1)(d+1)-2\Bigg(\displaystyle\sum_{t=1}^{k}L(T_{i})\Bigg)+1 \\
& = & \displaystyle\sum_{t=1}^{k}(n_{i}-1)(d-d_{i}+d_{i}+1)-2\Bigg(\displaystyle\sum_{t=1}^{k}L(T_{i})\Bigg)+1
\eean
\bean
& = & \displaystyle\sum_{i=1}^{k}\rn(T_{i})+\displaystyle\sum_{i=1}^{k}(n_{i}-1)(d-d_{i})-k+1.
\eean
\end{proof}

\begin{theorem}\label{thm:TSk} If $T \in \mathcal{T}_{lb}$ then $T_{S_{k}} \in \mathcal{T}_{lb}$, where $k \geq 3$ and
\begin{equation}\label{rn:TSk}
\rn(T_{S_{k}}) = k(\rn(T)+n_{0}(d-d_{0}-2)+d_{0})+1.
\end{equation}
\end{theorem}
\begin{proof} We prove $\rn(T_{S_{k}}) = lb(T_{S_{k}})$ and for this purpose, it is enough to show that there exists a linear order $u_{0},u_{1},...,u_{n-1}$ of vertices of $T_{S_{k}}$ which satisfies the conditions of Theorem \ref{thm:ub}.

We denote an internal vertex of $k$-star by $w$ and leaf vertices adjacent to $w$ by $x^{s}$, $s = 1,2,\ldots,k$. It is clear that $w$ is a weight center of $T_{S_{k}}$.

Since $T \in \mathcal{T}_{lb}$, the radio number of $T$ is given by
\be\label{rn2}
\rn(T) = (n_{0}-1)(d_{0}+1)-2L(T)+1.
\ee
Moreover, by Theorem \ref{thm:ub}, let $x^{s}_{t}$, $0 \leq t \leq n_{0}-1$ be a linear order of vertices of each copy of $T$ attached to $x^{s}$, $1 \leq s \leq k$ which satisfies the following conditions (a) $x^{s}_{0}$ = $w$ and $x^{s}_{n_{0}-1} \in N(w)$, $1 \leq s \leq k$, (b) $d_{T}(x^{s}_{l},x^{s}_{m}) \geq \sum_{t=l}^{m-1}(L_{T}(x^{s}_{t})+L_{T}(x^{s}_{t+1}))-(l-m)(d_{0}+1)+(d_{0}+1)$. The radio labeling defined by \eqref{eq:f0}-\eqref{eq:f} is an optimal radio labeling whose span is the right-hand side of \eqref{rn2}.

We define a linear order $u_{0},u_{1},...,u_{n-1}$ of vertices of $T_{S_{k}}$ as follows: Let $u_{0}$ = $w$ and for all other $u_{i}, 1 \leq i \leq n-1$ we use the following Algorithm 2.

\begin{algorithm}[h]
\caption{A linear order $\vec{u} := \{u_0,u_1,\ldots,u_{n-1}\}$ of $V(T_{S_{k}})$.}
\hspace*{\algorithmicindent} \textbf{Input:} A list of linear orders $u^{s}_{0},u^{s}_{1},\ldots,u^{s}_{n_{0}-1}$ of $s^{th}$ copy of $V(T)$ and a list $x^{s}, 1 \leq s \leq k$ with $w$.
\begin{algorithmic}[1]
\State $u_{0} \leftarrow w$
\For{\texttt{$1 \leq t \leq n_{0}-1$}}
\For{\texttt{$1 \leq s \leq k$}}
\State $i \leftarrow (t-1)k+s$
\State $u_{i} \leftarrow x^{s}_{t}$
\EndFor
\EndFor
\For{\texttt{$1 \leq s \leq k$}}
\State $i \leftarrow n-k-1+s$
\State $u_{i} \leftarrow x^{s}$
\EndFor
\State \textbf{return} $\vec{u}:=\{u_0,u_1,\ldots,u_{n-1}\}$
\end{algorithmic}
\hspace*{\algorithmicindent} \textbf{Output:} A linear order $\vec{u} = \{u_0,u_1,\ldots,u_{n-1}\}$.
\end{algorithm}

Then $u_{0}$ = $w$ and $u_{n-1} \in N(w)$ and for all $1 \leq i \leq n-2$, $u_{i}$ and $u_{i+1}$  are in different branches.

\textsf{Claim:} The linear order \{$u_{0},u_{1},...,u_{n-1}$\} satisfies \eqref{eq:dij}.

Let $u_{i}$ and $u_{j}$, $0 \leq i < j \leq n-1$ be two arbitrary vertices. Note that $d \geq d_{0}+2$ and $L_{T_{S_{k}}}(v)$ = $L_{T}(v)+1$, for any $v \in V(T_{S_{k}})$. We denote the right-hand side of \eqref{eq:dij} by $S_{i,j}$ for simplicity. If $u_{i},u_{j}$ are in different branches then $d_{T_{S_{k}}}(u_{i},u_{j})$ = $d_{T}(u_{i},u_{j})+2$. Hence, we have $S_{i,j}$ = $\sum_{t=i}^{j-1}(L_{T_{S_{k}}}(u_{t})+L_{T_{S_{k}}}(u_{t+1}))-(j-i)(d+1)+(d+1) \leq \sum_{t=i}^{j-1}(L_{T}(u_{t})+L_{T}(u_{t+1})+2)-(j-i)(d_{0}+3)+(d_{0}+3) = \sum_{t=i}^{j-1}(L_{T}(u_{t})+L_{T}(u_{t+1}))-(j-i)(d_{0}+1)+(d_{0}+1)+2 \leq d_{T}(u_{i},u_{j})+2 = d_{T_{S_{k}}}(u_{i},u_{j})$. If $u_{i}$ and $u_{j}$ are in the same branch of $T_{S_{k}}$ then note that $d_{T_{S_{k}}}(u_{i},u_{j})$ = $d_{T}(u_{i},u_{j})$ and $j-i$ = $k(l-m)$, where $k \geq 3$ and $l-m \geq 1$. Let $\alpha$ = max\{$L_{T_{S_{k}}}(u_{t})$ : $i \leq t \leq j$\} then we have, $S_{i,j}$ = $\sum_{t=i}^{j-1}(L_{T_{S_{k}}}(u_{t})+L_{T_{S_{k}}}(u_{t+1}))-(j-i)(d+1)+(d+1) \leq (k-1)(l-m)(2\alpha-d)+\sum_{t=l}^{m-1}(L_{T}(u_{t})+L_{T}(u_{t+1}))-(l-m)(d_{0}+1)+(d_{0}+1)+2-(k-1)(l-m) \leq d_{T}(u_{l},u_{m})+2-(k-1)(l-m) \leq d_{T}(u_{l},u_{m}) = d_{T_{S_{k}}}(u_{i},u_{j})$.

Thus, in each case above a linear order $u_{0},u_{1},...,u_{n-1}$ satisfies the conditions of Theorem \ref{thm:ub}.  The radio number for $T_{S_{k}}$ is given by the right-hand side of \eqref{eq:ub} for which $n$ = $kn_{0}+1$, $L(T_{S_{k}})$ = $k(L(T)+n_{0})$ and using \eqref{rn2} we have
\bean
\rn(T_{S_{k}}) & = & (n-1)(d+1)-2L(T_{S_{k}})+1 \\
& = & (kn_{0})(d+1)-2k(L(T)+n_{0})+1 \\
& = & k((n_{0}-1)(d_{0}+1)-2L(T))+kn_{0}(d-d_{0}-2)+k(d_{0}+1)+1 \\
& = & k(\rn(T)-1)+kn_{0}(d-d_{0}-2)+k(d_{0}+1)+1 \\
& = & k(\rn(T)+n_{0}(d-d_{0}-2)+d_{0})+1.
\eean
\end{proof}

\begin{theorem}\label{thm:TDk} If $T \in \mathcal{T}_{lb}$ then $T_{D_{k}} \in \mathcal{T}_{lb}$, where $k \geq 2$ and
\begin{equation}\label{rn:TDk}
\rn(T_{D_{k}}) = 2k(\rn(T)+n_{0}(d-d_{0}-3)+d_{0})+d.
\end{equation}
\end{theorem}
\begin{proof} We prove $\rn(T_{D_{k}}) = lb(T_{D_{k}})$ and for this purpose, it is enough to show that there exists a linear order $u_{0},u_{1},...,u_{n-1}$ of vertices of $T_{D_{k}}$ which satisfies the conditions of Theorem \ref{thm:ub}.

We denote two internal vertices of $k$-double star by $w_{1}$ and $w_{2}$ and leaf vertices adjacent to $w_{1}$ and $w_{2}$ by $x^{s}$, $s$ = 2,4,...,$2k$ and $x^{s}$, $s$ = 1,3,...,$2k-1$ respectively. It is clear that $w_{1}$ and $w_{2}$ are weight centers of $T_{D_{k}}$.

Since $T \in \mathcal{T}_{lb}$, the radio number of $T$ is given by
\be\label{rn3}
\rn(T) = (n_{0}-1)(d_{0}+1)-2L(T)+1.
\ee
Moreover, by Theorem \ref{thm:ub}, let $x^{s}_{t}$, $0 \leq t \leq n_{0}-1$ be a linear order of vertices of each copy of $T$ attached to $x^{s}$, $1 \leq s \leq 2k$ which satisfies the following conditions (a) $x^{s}_{0}$ = $w$ and $x^{s}_{n_{0}-1} \in N(w)$, $1 \leq s \leq 2k$, (b) $d_{T}(x^{s}_{l},x^{s}_{m}) \geq \sum_{t=l}^{m-1}(L_{T}(x^{s}_{t})+L_{T}(x^{s}_{t+1}))-(l-m)(d_{0}+1)+(d_{0}+1)$. The radio labeling defined by \eqref{eq:f0}-\eqref{eq:f} is an optimal radio labeling whose span is the right-hand side of \eqref{rn3}.

We define a linear order $u_{0},u_{1},...,u_{n-1}$ of vertices of $T_{D_{k}}$ as follows: Let $u_{0}$ = $w_{1}$, $u_{n-1}$ = $w_{2}$ and for all other $u_{i}, 1 \leq i \leq n-2$ we apply the following Algorithm 3.
\begin{algorithm}
\caption{A linear order $\vec{u} := \{u_0,u_1,\ldots,u_{n-1}\}$ of $V(T_{D_{k}})$.}
\hspace*{\algorithmicindent} \textbf{Input:} A list of linear orders $u^{s}_{0},u^{s}_{1},\ldots,u^{s}_{n_{0}-1}$ of $s^{th}$ copy of $V(T)$ and a list $x^{s}, 1 \leq s \leq 2k$ with $w_{1}$ and $w_{2}$.
\begin{algorithmic}[1]
\State $u_{0} \leftarrow w_{1}$
\State $u_{n-1} \leftarrow w_{2}$
\For{\texttt{$1 \leq t \leq n_{0}-1$}}
\For{\texttt{$1 \leq s \leq 2k$}}
\State $i \leftarrow (t-1)2k+s$
\State $u_{i} \leftarrow x^{s}_{t}$
\EndFor
\EndFor
\For{\texttt{$1 \leq s \leq 2k$}}
\State $i \leftarrow n-2k-1+s$
\State $u_{i} \leftarrow x^{s}$
\EndFor
\State \textbf{return} $\vec{u}:=\{u_0,u_1,\ldots,u_{n-1}\}$
\end{algorithmic}
\hspace*{\algorithmicindent} \textbf{Output:} A linear order $\vec{u} = \{u_0,u_1,\ldots,u_{n-1}\}$.
\end{algorithm}

Then $u_{0}$ = $w_{1}$ and $u_{n-1}$ = $w_{2}$ and for all $1 \leq i \leq n-2$, $u_{i}$ and $u_{i+1}$  are in opposite branches.

\textsf{Claim:} The linear order \{$u_{0},u_{1},...,u_{n-1}$\} satisfies \eqref{eq:dij}.

Let $u_{i}$ and $u_{j}$, $0 \leq i < j \leq n-1$ be two arbitrary vertices. Note that $d \geq d_{0}+3$ and $L_{T_{D_{k}}}(v)$ = $L_{T}(v)+1$ for any $v \in V(T_{D_{k}})$. We denote the right-hand side of \eqref{eq:dij} by $S_{i,j}$ for simplicity. If $u_{i}$ and $u_{j}$ are in opposite branches then $d_{T_{D_{k}}}(u_{i},u_{j}) = d_{T}(u_{i},u_{j})+3$. Hence, we have $S_{i,j}$ = $\sum_{t=i}^{j-1}(L_{T_{D_{k}}}(u_{t})+L_{T_{D_{k}}}(u_{t+1}))-(j-i)d+d+1 \leq \sum_{t=i}^{j-1}(L_{T}(u_{t})+L_{T}(u_{t+1})+2-(d_{0}+3))+d_{0}+4 = \sum_{t=i}^{j-1}(L_{T}(u_{t})+L_{T}(u_{t+1})-(d_{0}+1))+(d_{0}+1)+3 \leq d_{T}(u_{i},u_{j})+3 = d_{T_{D_{k}}}(u_{i},u_{j})$. If $u_{i}$ and $u_{j}$ are in different branches of $T_{D_{k}}$ then $d_{T_{D_{k}}}(u_{i},u_{j})$ = $L_{T_{D_{k}}}(u_{i})+L_{T_{D_{k}}}(u_{j})$ and $j-i \geq 2$. Let $\alpha = \max\{L_{T_{D_{k}}}(u_{t}) : i<t<j\}$ then we have, $S_{i,j}$ = $\sum_{t=i}^{j-1}(L_{T_{D_{k}}}(u_{t})+L_{T_{D_{k}}}(u_{t+1}))-(j-i)d+(d+1) =  L_{T_{D_{k}}}(u_{i})+L_{T_{D_{k}}}(u_{j})+2\sum_{t=i+1}^{j-1}L_{T_{D_{k}}}(u_{t})-(j-i-1)d+1 \leq L_{T_{D_{k}}}(u_{i})+L_{T_{D_{k}}}(u_{j})+(j-i-1)(2\alpha-(d-1))-(j-i-1)+1 \leq L_{T_{D_{k}}}(u_{i})+L_{T_{D_{k}}}(u_{j}) = d_{T_{D_{k}}}(u_{i},u_{j})$.
If $u_{i}$ and $u_{j}$ are in the same branch of $T_{D_{k}}$ then note that $d_{T_{D_{k}}}(u_{i},u_{j})$ = $d_{T}(u_{i},u_{j})$ and $j-i$ = $2k(l-m)$, where $k \geq 2$ and $l-m \geq 1$. Let $\alpha$ = max\{$L_{T_{D_{k}}}(u_{t})$ : $i \leq t \leq j$\} then we have, $S_{i,j}$ = $\sum_{t=i}^{j-1}(L_{T_{D_{k}}}(u_{t})+L_{T_{D_{k}}}(u_{t+1}))-(j-i)d+(d+1) \leq (2k-1)(l-m)(2\alpha-d)+\sum_{t=l}^{m-1}(L_{T}(u_{t})+L_{T}(u_{t+1}))-(l-m)(d_{0}+1)+(d_{0}+1)+3-(2k-1)(l-m) \leq d_{T}(u_{l},u_{m})+3-(2k-1)(l-m) \leq d_{T}(u_{l},u_{m}) = d_{T_{D_{k}}}(u_{i},u_{j})$.

Thus, in each case above a linear order $u_{0},u_{1},...,u_{n-1}$ satisfies the conditions of Theorem \ref{thm:ub}.  The radio number for $T_{D_{k}}$ is given by the right-hand side of \eqref{eq:ub} for which $n$ = $2kn_{0}+2$, $L(T_{D_{k}})$ = $2k(L(T)+n_{0})$ and using \eqref{rn3} we have
\bean
\rn(T_{D_{k}}) & = & (n-1)d-2L(T_{D_{k}}) \\
& = & (2kn_{0}+1)d-2(2kL(T)+2kn_{0}) \\
& = & 2k((n_{0}-1)(d_{0}+1)-2L(T))+2kn_{0}(d-d_{0}-3)+2k(d_{0}+1)+d \\
& = & 2k(\rn(T)-1)+2kn_{0}(d-d_{0}-3)+2k(d_{0}+1)+d \\
& = & 2k(\rn(T)+n_{0}(d-d_{0}-3)+d_{0})+d.
\eean
\end{proof}
In Table \ref{Table1}, a list of known lower bound trees with its radio number (which is useful for \eqref{rn:Twk}, \eqref{rn:TSk} and \eqref{rn:TDk}) is given. The reader may apply above described techniques on the lower bound trees and construct the large lower bound trees for the radio number. The procedure can be repeated on newly obtained lower bound trees again to form the large lower bound trees for the radio number.
\begin{table}[h]
\centering
\caption{A list of lower bound trees}\label{Table1}
\begin{tabular}{||l|l|l|l||}
  \hline\hline
  Sr. & Tree $T$ & $\rn(T)$ & Reference \\ \hline\hline
  1 & $P_{2k}$ & $2k(k-1)+1$ &  \cite{Liu} \\
  2 & $T_{h,m}$ & $\frac{m^{h+2}+m^{h+1}-2hm^2+(2h-3)m+1}{(m-1)^2}$ & \cite{Li} \\
  3 & $T^1$ & $(n-1)(d+1)+1-\dis\sum_{i=1}^{h}\(im_0\dis\prod_{0 < j < i}(m_j-1)\)$ & \cite{Tuza} \\
  4 & $T^2$ & $(n-1)d-4\dis\sum_{i=1}^{h}\(i\dis\prod_{j=0}^{i-1}(m_j-1)\)$ & \cite{Tuza} \\
  5 & $B(m,k)$ & $m(k+6)+1$ & \cite{Bantva2} \\
  6 & $F(m,k)$ & $\frac{(m^2+\ve)k}{2}+5m-3$ & \cite{Bantva2} \\
  7 & $C(2m,k)$ & $2(m-1)^2(k-1)+2m-1$ & \cite{Bantva2} \\
  \hline\hline
\end{tabular}
\end{table}

The readers may refer the following specific example for an illustration of the constructions described in this work and the procedure used in the proofs of Theorems \ref{thm:Twk} to \ref{thm:TDk}.

\begin{Example} In Fig. \ref{Fig1}, the lower bound tree $T_{w_2}$ is formed using lower bound trees $T_1$ and $T_2$ while $T_{S_3}$ and $T_{D_2}$ are formed by taking graph operations of the tree $T$ with $S_3$ and $D_2$.
\begin{figure}[h!]
\begin{center}
  \includegraphics[width=4.7in]{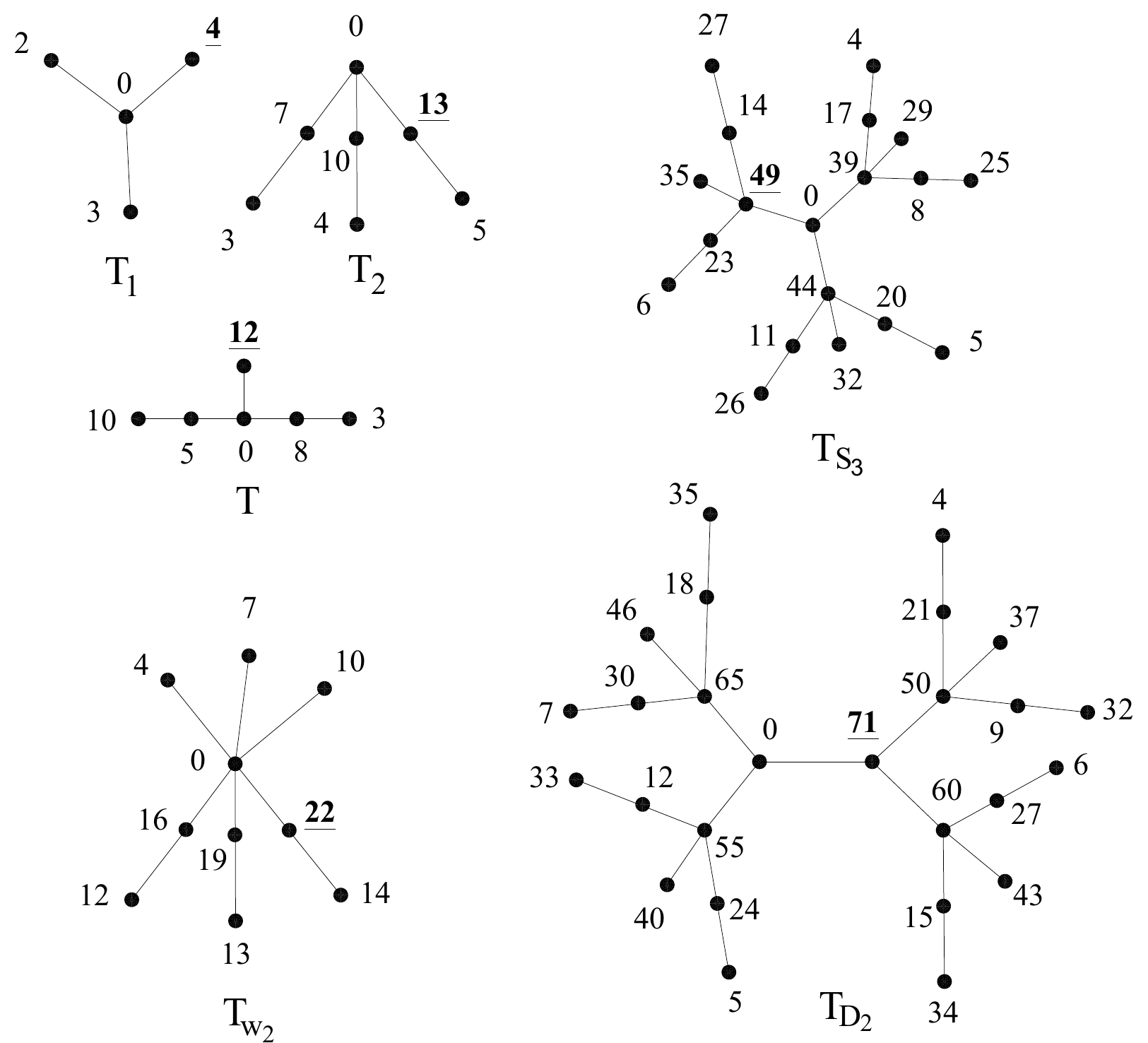}
  \caption{Optimal radio labeling for $T_{w_{2}}$, $T_{S_{3}}$ and $T_{D_{2}}$.}\label{Fig1}
\end{center}
\end{figure}
\end{Example}

\section{Concluding Remarks} We gave three techniques to find large lower bound trees which is obtained by taking graph operation on a known lower bound tree or a family of lower bound trees. The techniques can be repeated on newly obtained lower bound trees to produce more large lower bound trees. We related the radio number of newly obtained lower bound trees with given lower bound tree or a family of trees.

\section*{Acknowledgement} The research of the first author is supported by Research Promotion under Technical Education - STEM research project grant of Government of Gujarat.


\begin{thebibliography}{00}
\bibitem{Bantva2}
D. Bantva, S. Vaidya and S. Zhou, Radio number of trees, Discrete Applied Math., 217 (2016), 110--122.

\bibitem{Bantva3}
D. Bantva, Furthur results on the radio number of trees, Electronic Notes in Discrete Math., 63 (2017), 85--91.

\bibitem{Calamoneri}
T. Calamoneri, The $L(h,k)$-labeling problem: An updated survey and annotated bibliography, The Computer Journal, 54(8) (2011), 1344--1371.

\bibitem{Chartrand}
G. Chartrand and P. Zhang, Radio colorings of graphs - a survey, Int. J. Comput. Appl. Math., 2(3) (2007), 237--252.

\bibitem{Chartrand1}
G. Chartrand, D. Erwin, F. Harary and P. Zhang, Radio labelings of graphs, Bull. Inst. Combin. Appl., 33 (2001), 77--85.

\bibitem{Chartrand2}
G. Chartrand, D. Erwin and P. Zhang, A graph labeling suggested by FM channel restrictions, Bull. Inst. Combin. Appl., 43 (2005), 43--57.

\bibitem{Chavez}
A. Chavez, D. Liu and M. Shurman, Optimal radio-$k$-labelings of trees, European Journal of Combi., 91 (2021), 130203.

\bibitem{Griggs}
J. R. Griggs and R. K. Yeh, Labeling graphs with condition at distance 2, SIAM J. Discrete Math., 5(4) (1992), 586--595.

\bibitem{Tuza}
V. Hal\'asz and Z. Tuza, Distance-constrained labeling of complete trees, Discrete Math., 338 (2015), 1398-1406.

\bibitem{Hale}
W. K. Hale, Frequency assignment: Theory and applications, Proc. IEEE 68(12) (1980), 1497--1514.

\bibitem{Li}
X. Li, V. Mak and S. Zhou, Optimal radio labelings of complete $m$-ary trees, Discrete Applied Math., 158 (2010), 507--515.

\bibitem{Daphne1}
D. Liu, Radio number for trees, Discrete Math., 308 (2008), 1153--1164.

\bibitem{Liu}
D. Liu and X. Zhu, Multi-level distance labelings for paths and cycles, SIAM J. Discrete Math., 19 (2005), 610--621.

\bibitem{West}
D. B. West, Introduction to Graph Theory, Prentice-Hall of India, 2001.

\bibitem{Yeh1}
R. K. Yeh, A survey on labeling graphs with a condition at distance two, Discrete Math., 306 (2006), 1217--1231.
\end{thebibliography}
\end{document}